\documentclass[10pt,a4]{article}
\usepackage{latexsym}
\usepackage{amsmath}
\usepackage{amssymb}
\usepackage{amsthm}
\usepackage{amscd}
\usepackage{mathrsfs}
\usepackage[all]{xy}

\newtheorem{theorem}{Theorem}[section]
\newtheorem{lemma}[theorem]{Lemma}
\newtheorem{corollary}[theorem]{Corollary}
\newtheorem{proposition}[theorem]{Proposition}

\theoremstyle{definition}

\newtheorem{remark}[theorem]{Remark}
\newtheorem{definition}[theorem]{Definition}

\theoremstyle{remark}
\newtheorem{notation}{Notation}

\makeatletter

\@addtoreset{equation}{section}
\makeatother

\renewcommand{\eqref}[1]{(\ref{#1})}

\renewcommand{\bigskip}{\vspace{0.2cm}}

\begin{document}

\title{{\Large 
\textbf{On the cohomology of 
the Lubin-Tate curve 
of level $2$
 and the Lusztig theory 
 over finite rings}}}
\maketitle

\begin{center}
\textbf{ Tetsushi Ito, Yoichi 
Mieda and Takahiro Tsushima}
\end{center}

\maketitle
\begin{abstract}
In \cite{IMT}, 
we study  
the \'{e}tale 
cohomology
group $W$ 
of irreducible
 components, 
with each 
having 
an affine 
model
 $X^q+X=\xi Y^{q+1}+c$
 with $\xi \in
  \mathbb{F}^{\times}_q$ 
  and $c \in \mathbb{F}_q,$ 
in the stable
 reduction 
 of 
the Lubin-Tate 
curve 
$\mathcal{X}
(\pi^2)$. 
Then, 
$W$ is considered as a 
${\rm GL}_2
(\mathcal{O}_F/\pi^2)$
-representation and
 relates to
unramified 
cuspidal representations 
of ${\rm GL}_2(F)$ of level $1.$
In this paper, 
we study a relationship 
between 
$W$ and the 
\'{e}tale cohomology groups
of some
{\it Lusztig varieties} 
costructed 
in \cite{Lus} and \cite{Lus2}.
\end{abstract}
\section{Introduction}
Let $F$ be a non-archimedean local field 
with ring of integers $\mathcal{O}_F$, 
uniformizer $\pi$ and residue field 
$\mathbb{F}_q$ of 
characteristic $p>0.$
We write $\mathbf{F}$ for 
an algebraic closure 
of $\mathbb{F}_q.$
Let 
$\bar{F}$ denote 
an algebraic closure  
of $F$ and $\mathbf{C}$ 
its completion.  
Roughly speaking, the Lubin-Tate curve 
$\mathcal{X}(\pi^n)$
is the generic fiber of the
 deformation space 
of a formal $\mathcal{O}_F$-module 
of height $2$
over $\mathbf{F}$ 
equipped with the Drinfeld 
level $\pi^n$-structure. 
Namely,  
$\mathcal{X}(\pi^n)$
is a rigid analytic curve.
In \cite{IMT}, 
we compute 
defining equations 
of irreducible 
components 
in the stable 
reduction of
 the Lubin-Tate 
 curve 
 $\mathcal{X}(\pi^2)$, 
 and study 
 the \'{e}tale 
 cohomology group of
  some components 
as a representation of 
a product of 
$G_2^F:={\rm GL}_2
(\mathcal{O}_F/\pi^2)$,
the central 
division algebra 
$\mathcal{O}^{\times}_D$ 
over $F$ of invariant 
$1/2$, and 
the inertia group $I_F.$
For $n \geq 1,$ 
let $U_D^n$ be
 open compact subgroups
of $\mathcal{O}_D^{\times}.$
We put $\mathcal{O}_n^{\times}
:=\mathcal{O}_D^{\times}/U_D^n$
and $\mathbf{G}:
=G_2^F \times 
\mathcal{O}_3^{\times} \times I_F.$

 We set
 \[
 \mathfrak{S}
 ^{\mathbb{F}^{\times}
 _{q}}_{00}:=\{
 (x_0,y_0) \in \mathbf{F}^2\ |\ 
 \xi:=x_0^qy_0-x_0y_0^q \in 
 \mathbb{F}
 ^{\times}_q,\
  x_0^{q^2-1}
  =y_0^{q^2-1}=-1
 \}.
 \]
 Then, we have 
 $|\mathfrak{S}
 _{00}
 ^{\mathbb{F}^{\times}_q}|
 =|{\rm GL}_2(\mathbb{F}_q)|
 =q(q-1)(q^2-1).$
 For  $i=(x_0,y_0) 
 \in \mathfrak{S}
 ^{\mathbb{F}^{\times}_q}_{00},$ 
 let $X_i$
 denote the 
 smooth compactification 
 of an affine 
 curve 
 $X^{q^2}-X=\xi
 (Y^{q(q+1)}-Y^{q+1}).$
 Each curve $X_i$ has 
 $q$ connected components and 
 each component with genus $q(q-1)/2$
  has an affine model
 $X^q+X=\xi Y^{q+1}+c$ 
 with some 
 $c \in \mathbb{F}_q.$
 In the stable reduction of 
 $\mathcal{X}(\pi^2)$, the 
 components 
 $\{X_i\}
 _{i \in 
 \mathfrak{S}
 ^{\mathbb{F}
 ^{\times}_q}_{00}}$ 
 appear,
  which is proved 
  in \cite{IMT}.
We set 
\[
W:=\bigoplus_
{i \in 
\mathfrak{S}
^{\mathbb{F}^{\times}_q}_{00}}
H^1(X_i,
\overline{\mathbb{Q}}_l) 
\subset 
H^1(\mathcal{X}(\pi^2)_{\mathbf{C}},
\overline{\mathbb{Q}}_l)
\]
with a prime number $l \neq p.$
Apriori, the 
prodcut group 
$\mathbf{G}$ 
acts on the curve
 $\mathcal{X}(\pi^2)_{\mathbf{C}}$
 on the right,
  and hence
 on the \'{e}tale
  cohomology group 
 $H^1(\mathcal{X}(\pi^2)_{\mathbf{C}},
\overline{\mathbb{Q}}_l)$
on the left.
See \cite[Section 1]{Ca} 
for more details.
Then, the subspace $W$ is stable under 
the 
$\mathbf{G}$-action.
In \cite[7.2]{IMT}, 
we analyze 
$W$ as a 
$\mathbf{G}$-representation.
Let $E/F$ denote 
the unramified quadratic extension.
Canonically, 
the group 
$\mathcal{O}^{\times}_3$
contains
 $\Gamma:
 =(\mathcal{O}_E/\pi^2)^{\times}$
as a subgroup.
For any finite 
extension $L/F$, 
let $W_L$ denote 
the Weil group of 
$L$ and $I_L$ the 
inertia subgroup 
of $W_L.$
 Then, we have the 
Artin reciprocity map 
$\mathbf{a}_L:W^{\rm ab}_{L} 
\overset{\sim}{\to} L^{\times}$
normalized 
such that
 the geometric Frobenius
goes to a 
prime element under this map.
We have the 
restriction $\mathbf{a}_L:
I_L^{\rm ab} 
\to \mathcal{O}^{\times}_L.$
Hence, in particular, 
the map 
$\mathbf{a}_E$ 
induces the 
following map
$I_F \to 
I_F^{\rm ab} 
\simeq I_E^{\rm ab} 
\overset{\mathbf{a}_{E}
|_{I^{\rm ab}_E}}
{\longrightarrow }
 \mathcal{O}^{\times}_E \to 
\Gamma$, which we denote by
 $\mathbf{a}_E.$
Then, 
the inertia group
 $I_F$
acts on $W$ by 
factoring through 
the map 
$\mathbf{a}_E:I_F \to 
\Gamma.$
Hence, the restriction 
 $W|_{G_2^F \times \{1\} 
\times I_F}$
can be considered as a 
$G_2^F \times \Gamma$-representation,
which we denote by $W_1$.
Let $W_2$ denote the restriction
 $W|_{G_2^F \times \Gamma \times \{1\}}.$

Let $G$ be  
 a connected 
 reductive
algebraic 
 group over 
 $\mathbf{F}$
 with a given
  $\mathbb{F}_q$-rational 
  structure
 with associated 
 Frobenius 
 morphism $F:G \to G.$
For 
 $n \geq 1$, 
 in \cite{Lus}
 and \cite{Lus2}, 
G.\ Lusztig constructs 
an affine algebraic 
 variety over $\mathbf{F}$,  
with 
$G(\mathcal{O}_F/\pi^n)$-action
 and
 some torus 
 $T(\mathcal{O}_F/\pi^n)$-action, 
whose \'{e}tale 
cohomology group  
realizes some 
irreducible 
representations of 
$G(\mathcal{O}_F/\pi^n).$ 
These representations 
are attached 
to characters of 
$T(\mathcal{O}_F/\pi^n)$
 in 
general position.
The works 
\cite{Lus} and \cite{Lus2} 
are generalizations
 of the 
 Deligne-Lusztig theory 
  \cite{DL}
over finite fields 
to finite rings 
$\mathcal{O}_F/\pi^n.$
In this paper, we call
 the variety over $\mathbf{F}$
{\it  the Lusztig variety 
for $G(\mathcal{O}_F/\pi^n).$}
Specialized to 
a case 
$S_2^F:={\rm SL}_2
(\mathcal{O}_F/\pi^2),$
 Lusztig gives an 
 explicit description 
 of the Lusztig 
 variety of dimension $2$
  for $S_2^F$
 and studies 
 its cohomology 
 groups 
 in \cite[Section 3]{Lus}.
 In the cohomology groups, all 
 {\it unramified} representations, 
 in a sense of \cite[p.37]{Sha},
  occur each one 
  with multiplicity $2$. 
  Hence, we call 
  the variety the 
  unramified 
  Lusztig surface for $S_2^F.$
 
In this paper, we will 
study the
 affine unramified 
 Lusztig surface 
 $\Tilde{X}$
  for $G_2^F$ 
  in the same way 
  as in
   \cite[Section 3]{Lus}. 
  The surface $\Tilde{X}$ 
  admits an 
  action
   of a product group 
   $G_2^F \times \Gamma$.
 Then, 
 we investigate 
  the \'{e}tale 
  cohomology 
  group 
  $H_c^2(\Tilde{X},
  \overline{\mathbb{Q}}_l)$ as 
  a $G_2^F \times 
  \Gamma$-representation
   in the same way as in loc.\ cit.   
As a result, 
we show that 
a direct sum $\rho_{\rm DL}$ 
of all {\it cuspidal} or 
{\it unramifed}
representations of $G_2^F$, 
in a sense of \cite{Sta},  
appears 
as a 
$G_2^F$-subrepresentation 
of $H_c^2(\Tilde{X},
  \overline{\mathbb{Q}}_l)$.
By using
 the explicit description 
 of $H_c^2(\Tilde{X},
 \overline{\mathbb{Q}}_l)$
as a $G_2^F$-representation 
and the analysis of $W$
in \cite{IMT}, 
 we prove that 
the $G_2^F \times
 \Gamma$-representations 
$W_1$ and $W_2$ 
mentioned above
are written with respect to
 $\rho_{\rm DL}.$  
  See
   Proposition \ref{gey}
 for precise 
 statements. 
 In other words, 
 the representation $W$
  gives another  
  cohomological
   construction 
  for unramfied 
  irreducible 
  representations 
  of $G_2^F.$
For $S^F_2,$ 
we 
obtain 
the same things 
as $G_2^F.$
See Remarks \ref{s_1}
 and \ref{s_2}  
for more details.
In section \ref{5}, 
we write down 
the Lusztig curve
 $X_D$
for 
$\mathcal
{O}_3^{\times}$
and study its 
cohomology
 group $H_c^1$
 as a
  $\mathcal{O}_3^{\times}
   \times \Gamma$-representation.
 The Lusztig variety 
 for the group of 
 reduced norm one in a 
 central division 
 algebra of invariant $1/n$
  over $F$
  is explicitly 
  described 
  in \cite[Section 2]{Lus}.
We consider the product 
$\mathbf{X}:=\Tilde{X} \times X_D$.
Then, this variety $\mathbf{X}$
admits an action of 
 $G_2^F \times 
 \mathcal{O}_3^{\times} 
 \times \Gamma.$
 Hence, the \'{e}tale cohomology 
 group $H_c^3(\mathbf{X},
 \overline{\mathbb{Q}}_l)$
 is considered as a 
 $\mathbf{G}$-representation
 through a surjective 
 map $\mathbf{G} \to 
 G_2^F \times \mathcal{O}_3^{\times} 
 \times \Gamma.$
As a result of 
the analyses of 
$H_c^1(X_D,
\overline{\mathbb{Q}}_l)$ as a 
$\mathcal{O}_3^{\times}
 \times \Gamma$-representation 
 in Lemma \ref{lq1}
and $H_c^2(\Tilde{X},
\overline{\mathbb{Q}}_l)$
as a $G_2^F \times \Gamma$-representation
in Proposition \ref{sp},
we show that $W$ is contained in 
$H_c^3(\mathbf{X},\overline{\mathbb{Q}}_l)$
as a 
$\mathbf{G}$-subrepresentation
in Theorem \ref{fin}.
 See also 
 \cite{Y} for a 
 connection 
 between the 
 non-abelian 
 Lubin-Tate 
 theory of level $1$ 
 and the 
 Deligne-Lusztig
 theory over finite fields. 
\begin{notation}
We fix some 
notations used 
throughout 
this paper.
Let $F$ be a 
non-archimedean 
local field with 
ring of integers 
$\mathcal{O}_F$, 
uniformizer 
$\pi$ and residue field 
$\mathbb{F}_q$ of 
characteristic $p>0.$
We identify $(\mathcal{O}_F
/\pi^2)^{\times}$
with $\mathbb{F}^{\times}_q 
\times \mathbb{F}_q$ by 
$a_0+a_1\pi \mapsto 
(a_0,(a_1/a_0)).$
Let $E/F$ be the unramified 
quadratic extension and 
$\mathcal{O}_E$ 
the ring of integers.
Of course, 
the residue field of $E$
is equal to 
$\mathbb{F}_{q^2}.$
Let $\tau \in 
{\rm Gal}(E/F)$ be 
a non-trivial element.
Let $\mathbf{F}$ 
denote an
 algebraic closure 
 of $\mathbb{F}_q.$
We set $G_n^F:
={\rm GL}_2
(\mathcal{O}_F/\pi^n)$
and $S_n^F:={\rm SL}_2
(\mathcal{O}_F/\pi^n)$
for $n \geq 1.$
Let $D$ be the 
central division algebra 
over $F$ of 
invariant $1/2$.
Let $\mathcal{O}_D$ 
denote the ring of 
integers of $D$, 
and $\varphi$
a prime element 
of $D$ such that $\varphi^2=\pi.$
Then, we have 
$\mathcal{O}_D=\mathcal{O}_E
 \oplus 
 \varphi \mathcal{O}_E$
 with $\varphi a=a^{\tau} \varphi$
 for $a \in \mathcal{O}_E.$
Furthermore, for $n \geq 1,$ 
let $U_D^n$ denote 
the open compact
 subgroups $1+(\varphi^n)$
  of $D^{\times}.$
 We set 
 $\mathcal{O}_n^{\times}:
 =\mathcal{O}^{\times}_D/U_D^n.$
 Let ${\rm Nrd}_{D/F}
 :D^{\times} \to 
 F^{\times}$
 denote the 
 reduced norm of $D^{\times}.$
For any finite 
extension $L/F$, 
let $W_L$ denote 
the Weil group of 
$L$ and $I_L$ the 
inertia subgroup 
of $W_L.$
 Then, we have the 
Artin reciprocity map 
$\mathbf{a}_L:W^{\rm ab}_{L} 
\overset{\sim}{\to} L^{\times}$
normalized such that
 the geometric Frobenius
goes to a 
prime element under this map.
Then, 
we have the 
restriction 
$\mathbf{a}_L:
I_L^{\rm ab} 
\to \mathcal{O}
^{\times}_L.$
For a finite extension $L/K$, 
we 
have $\mathbf{a}_K={\rm Nr}_{L/K}
 \circ \mathbf{a}_L.$
If $X$ is an affine algebraic variety
over $\mathbf{F}$ and $r \geq 1$, 
we set $X_r:=X(\mathbf{F}[[\pi]]/(\pi^r)).$
For a prime number $l \neq p$ and 
 a finite abelian group $M$, we write 
$M^{\vee}$ for a character group ${\rm Hom}
(M,\overline{\mathbb{Q}}^{\times}_l).$
For a representatin $\pi$ of a finite group 
$G$ over $\overline{\mathbb{Q}}_l,$
we write $\pi^{\vee}$ for its dual.
For any variety $Y/\mathbf{F}$ with 
an action of a
 finite abelian group 
and any character 
$w$ of that finite group,
let $H_c^j
(Y,
\overline
{\mathbb{Q}}_l)_w$ 
denote the subspace of
 $H_c^j(Y,\overline{\mathbb{Q}}_l)$
on which the finite group acts according to 
$w$.
\end{notation}

\section{Cuspidal representation of
$G_2^F$
}\label{Go}
In this section, 
we briefly recall definitions of
cuspidal 
representation of
$G_2^F,$ and 
strongly primitive character
of a torus $\Gamma:
=(\mathcal{O}_E/\pi^2)^{\times}$ 
in Definitions 
\ref{cus} and \ref{str}, which 
 are due to  \cite[5.1 and 5.2]{AOPS}, 
 \cite[6.2]{Onn},  
 \cite[p.2835]{Sta} and
  \cite{Sta2}. 
A cuspidal
representation of $G^F_2$
has degree $q(q-1)$, and 
there exist $q(q-1)(q^2-1)/2$ 
cuspidal representations 
up to equivalence.
These cuspidal 
representations
are parametrized 
by strongly 
primitive 
characters of 
a maximal 
torus $\Gamma 
\subset G^F_2.$
See Remark \ref{wq}.3
 for more details.

First, 
we prepare some notations.
Let $m:G^F_2 \to G^F_1$
be the canonical surjection.
We write $\bar{g}$ 
for the image of 
$g \in G_2^F$
by the 
map $m$.
We set $N:={\rm Ker}\ m.$
Then, the
 group
 $N$ is isomorphic
  to an abelian group 
${\rm M}_2(\mathbb{F}_q).$
All irreducible representations of 
$G^F_2$ occur as an 
irreducible component
of the induced 
representation 
${\rm Ind}_N^{G^F_2}(\Phi)$
with some 
character $\Phi \in N^{\vee}$.
We fix a non-trivial 
character $\eta 
\in \mathbb{F}^{\vee}_q.$
Then, we have the 
following isomorphism
\begin{equation}\label{do_n}
{\rm M}_2(\mathbb{F}_q) 
\overset{\sim}{\longrightarrow} N^{\vee}\ ;\
\beta \mapsto 
(\eta_{\beta}:h \mapsto \eta({\rm Tr}
_{{\rm M}_2(\mathbb{F}_q)/\mathbb{F}_q}(
\beta\pi^{-1}(h-1))).
\end{equation}
For example, see \cite[Section 2]{Sta2} 
for more details. 
The group $G^F_2$ 
acts on $N^{\vee}$ by 
conjugation.
For $\Phi \in N^{\vee},$ 
let $T(\Phi)$ denote the 
centralizer of $\Phi$
in $G_2^F$. 
Then, $T(\Phi)$ 
is a torus of $G^F_2.$
Let $N(\Phi):=T(\Phi)N.$ Then, 
$N(\Phi)$ is the stabilizer of $\Phi$ in $G_2^F.$
By the Clifford theory, 
any irreducible representation
$U$ of $N(\Phi)$ such that $\Phi \subset U|_{N}$
is of the form $\Psi_{\psi,\Phi}:$
$$\Psi_{\psi,\Phi}(tn)=\psi(t)\Phi(n)$$
for $t \in 
T(\Phi),n \in N$ 
where $\psi \in T(\Phi)^{\vee}$
and satisfies $\psi|_{T(\Phi) \cap N}
=\Phi|_{T(\Phi) \cap N}.$ 
See \cite[Theorem 2.1]{Sta2} for more
 details.
\begin{definition}\label{cus}
Let $\rho$ be an irreducible
 representation of $G^F_2.$
We assume that 
$\rho$ does not 
factor through $G^F_1.$
We call $\rho$ {\it cuspidal} if 
$\rho$ is equivalent to 
${\rm Ind}_{N(\Phi)}^{G_2^F}(\Psi_{\psi,\Phi})$
with
$\Phi=\eta_{\beta}$ in
 (\ref{do_n}), and 
$\beta$ is conjugate 
to the following
\begin{equation}\label{re2}
\left(
\begin{array}{cc}
0 & 1 \\
-\Delta & s
\end{array}
\right)
\end{equation}
where 
$X^2-sX+\Delta$ 
is an 
irreducible polynomial
in $\mathbb{F}_q[X]$.
\end{definition}
We also call 
cuspidal representation of $G_2^F$
{\it unramified}.
\begin{definition}(\cite[5.1]{AOPS})
\label{str}
We call a character $\psi$ of 
$\Gamma=(\mathcal{O}_E/\pi^2)^{\times}$
{\it strongly primitive} if
the restriction of 
$\psi$ to a subgroup
$\mathbb{F}_{q^2} \simeq 
{\rm Ker}\ (\Gamma
\to (\mathcal{O}_E/\pi)^{\times})$
does not factor 
through the trace map 
${\rm Tr}
_{\mathbb{F}_{q^2}/\mathbb{F}_q}:
\mathbb{F}_{q^2}
\to \mathbb{F}_q.$
\end{definition}

\begin{remark}\label{wq}
1.\ For ${\rm SL}_2,$ a similar definition 
to Definition \ref{cus}
is found in \cite[p.37]{Sha},
\\2.\ The dimension of the 
cuspidal representation of $G^F_2$
is equal to $q(q-1),$ 
because we have $[G^F_2:N(\Phi)]=q(q-1)$ if 
$\Phi$ is conjugate to (\ref{re2}). 
\\3.\ The definition \ref{cus} 
of cuspidal representation of 
$G_2^F$
above is equivalent 
to a definition of 
{\it strongly cuspidal} 
representation 
in \cite{Onn} or \cite{AOPS}. 
Descriptions of 
strongly cuspidal representations
are given in \cite[6.2]{Onn} 
and \cite[5.2]{AOPS}.
Then, as stated in \cite[Theorem B]{AOPS},
there is a canonical 
bijective correspondence
between strongly cuspidal 
representations of $G^F_2$
and ${\rm Gal}(E/F)$-orbits 
of strongly primitive 
characters of 
$\Gamma$. 
Therefore, the number of 
cuspidal representations
 matches the number 
of Galois orbits 
of strongly primitive characters of
$\Gamma$. 
Hence, there exist 
$q(q-1)(q^2-1)/2$ 
cuspidal representations up to equivalence,
which is pointed out 
in \cite[6.2]{Onn}.
\end{remark}

\section{Unramified 
Lusztig surface for 
$G_2^F$ 
and its cohomology 
$H_c^2$}\label{Lo}
In this section, 
we recall 
the Lusztig theory 
in \cite{Lus}
 and \cite{Lus2},
specialized to the case 
$G_2^F$.
The papers 
\cite{Lus} and \cite{Lus2}
are generalizations of 
the Deligne-Lusztig 
theory over a finite field
to 
finite rings 
$\mathcal{O}_F/\pi^r$ 
with $r \geq 2$.
Let $G$ be 
a reductive connected 
algebraic 
group 
over 
$\mathbf{F}$ with 
associated 
Frobenius map 
$F:G \to G$
and 
$T$ is a 
maximal 
torus of $G$.
Roughly speaking, 
the Lusztig variety for  
$G(\mathcal{O}_F/\pi^r)$
means some variety over $\mathbf{F}$ 
with $G(\mathcal{O}_F/\pi^r) 
\times T(\mathcal{O}_F/\pi^r)$-action, 
and 
its $\ell$-adic cohomology 
groups
give a 
correspondence between
characters of 
$T(\mathcal{O}_F/\pi^r)$
in general position 
and some 
irreducible 
representations 
of $G(\mathcal{O}_F/\pi^r).$ 
A concrete description 
of the Lusztig surface 
for $S_2^F$ 
is given in \cite[section 3]{Lus}.
In this section, 
in the same way as in loc.\ cit.,
we explicitly 
describe the 
Lusztig surface
for $G_2^F$ and 
the 
\'{e}tale cohomology
 group $H_c^2$
of it. 
The cohomology group
realizes all 
unramifed or cuspidal 
representations 
of $G_2^F$ 
introduced 
in Definition 
\ref{cus}.
See Proposition \ref{sp} 
and Corollary \ref{Ls}
for more details.
Arguments in this section are 
almost same as the ones in loc.\ cit.

In the following, 
we assume that 
$p$ is odd and 
\begin{equation}\label{asd}
{\rm char}\ F=p>0,\ 
{\rm or},\ 
{\rm char}\ F=0\ {\rm and}\ 
e_{F/\mathbb{Q}_p} \geq 2.
\end{equation}
Let $F^{\rm ur}$
be the maximal unramified 
extension of $F$ 
in a fixed
 algebraic 
 closure 
 of $F.$
$A:=\mathcal{O}
_{F^{\rm ur}}
/(\pi^2)$.
Then, by 
the assumption (\ref{asd}), 
we have $A \simeq 
\mathbf{F}[[\pi]]/(\pi^2).$
Define $F:A \to A$ by 
$F(a_0+a_1\pi)=a_0^q+a_1^q\pi$
where $a_0,a_1 \in \mathbf{F}.$
 Let $V$ be a $2$-dimensional 
$\mathbf{F}$-vector space with a fixed 
$\mathbb{F}_q$-rational
structure with the 
Frobenius 
map $F:V \to V.$ 
Let $e,e'$ be a basis of $V$.
 Let 
 $G:={\rm GL}(V)$.
 We identify 
 $G \simeq {\rm GL}_2(\mathbf{F})$
 by $g \mapsto 
 {\footnotesize 
 \left(
\begin{array}{cc}
a & b \\
c & d
\end{array}
\right)}
$ with $eg=ae+ce',e'g=be+de'.$
Then, $G$ also 
has a 
$\mathbb{F}_q$-rational 
structure
with the 
Frobenius map $F:G \to G$
such that $F(vg)=F(v)F(g)$ for
$g \in G, v \in V.$
We put $V_2:=V 
\otimes_{\mathbf{F}}A.$
Clearly, $V_2$ is a 
free $A$-module of rank $2$.
Let $T:=\{{\footnotesize
 \left(
\begin{array}{cc}
F(a) & 0 \\
0 & a
\end{array}
\right)}
\} \subset G$
 and $U:=\{{\footnotesize 
 \left(
\begin{array}{cc}
1 & c \\
0 & 1
\end{array}
\right)}
\} \subset G$.
Let $v:={\footnotesize 
\left(
\begin{array}{cc}
0 & 1 \\
-1 & 0
\end{array}
\right)}.
$

We define as follows
\begin{equation}\label{Ld}
\Tilde{X}:=\{g \in G_2\ |\ F(g)g^{-1} 
\in U_2v\},
\end{equation}
which we call the 
{\it 
unramified Lusztig 
surface for $G_2^F.$}
Then, $G_2^F$ acts on $\Tilde{X}$
on the {\it right.}
This definition 
is not equal to 
the one 
in 
\cite[3.2]{Lus}.
To obtain a 
left $G_2^F$-action 
on the cohomology group of 
$\Tilde{X},$
 in (\ref{Ld}),
 we slightly modify
 the definition 
 in loc.\ cit.\ 
 By a direct computation, 
 we check that, 
 for $g \in G_2^F,$
 the condition
$F(g)g^{-1} \in U_2v$ is equivalent to
$g=
{\footnotesize 
\left(
\begin{array}{cc}
-F(\mathbf{x}) & -F(\mathbf{y}) \\
\mathbf{x} & \mathbf{y}
\end{array}
\right)}
$
with some 
$\mathbf{x},\mathbf{y} \in A$
and ${\rm det}(g)
 \in (\mathcal{O}_F/\pi^2)^{\times}.$
We set
\begin{equation}\label{t}
\Gamma:=\{t \in T_2\
 |\ tF(t) \in 
 T_1^F\}.
 \end{equation}
 Let $\Gamma \ni t$
  act on $\Tilde{X}$
 by $g \mapsto t^{-1}g.$
Then, the group 
$\Gamma$ is isomorphic to
the following
$\{(a_0,a_1) 
\in \mathbf{F}^2\ |\ a_0 
\in \mathbb{F}
^{\times}_{q^2},\
 a_0^qa_1+a_0a_1^q 
 \in \mathbb{F}_{q}\}.$
Hence,
 $\Gamma$
 is isomorphic 
to a group 
$(\mathcal{O}_E
/\pi^2)^{\times} 
\ni a_0+a_1\pi$
of order $q^2(q^2-1).$
In the following, we fix 
 the identification
$\Gamma \simeq 
(\mathcal{O}_E/\pi^2)^{\times}.$
We define a 
subgroup
$A':=
\{1+\pi a_1\ 
|\ a_1^q+a_1=0\}
 \subset {\Gamma}.$
We put  
$$\mathfrak{S}:=
\{x \in V_2\ |\ x 
\wedge F(x) 
\in 
(\mathcal{O}_F/\pi^2)^{\times}
(e \wedge e')\}$$
with a 
 right
 $G_2^F$-action
  $x \mapsto xg.$ 
The group $\Gamma$ acts 
on $\mathfrak{S}$ by the inverse of 
scalar multiplication.
Then, 
as in \cite[3.4]{Lus}, 
we have 
an isomorphism 
$\iota:\Tilde{X} 
\overset{\sim}
{\longrightarrow} 
\mathfrak{S}\ ;\ g
 \mapsto e'g,$
  which is 
  compatible 
  with the 
  $G_2^F\ 
  \times 
  \Gamma$-actions.
We write 
$x=\mathbf{x}e
+\mathbf{y}e' 
\in \mathfrak{S}$
with $\mathbf{x}=x_0+x_1\pi$
and $\mathbf{y}=y_0+y_1\pi.$
Then,
we have the following 
isomorphism by 
a direct computation
\begin{equation}
\label{coo1}
\mathfrak{S} \simeq 
\{(x_0,x_1,y_0,y_1) \in 
\mathbf{F}^{4}\ |\ 
x_0y_0^q-x_0^qy_0 \in 
\mathbb{F}^{\times}_q,\ 
x_1y_0^q+x_0y_1^q
-x_0^qy_1-x_1^qy_0 
\in \mathbb{F}_q\}.
\end{equation}
Let $\lambda=a_0+\pi a_1 
\in \Gamma.$
Then, $\lambda^{-1}$ acts 
on $\mathfrak{S}$
as follows
\begin{equation}\label{ooi1}
\lambda^{-1}:\ 
(x_0,x_1,y_0,y_1) 
\mapsto (a_0x_0,a_1x_0+a_0x_1,
a_0y_0,a_1y_0+a_0y_1).
\end{equation}
Let
$g={\footnotesize \left(
\begin{array}{cc}
a_0+a_1\pi & b_0+b_1\pi \\
c_0+c_1\pi & d_0+d_1 \pi
\end{array}
\right)}
\in G_2^F.$
Then, the
 element $g$ acts
  on $\mathfrak{S}$ 
as follows
\begin{equation}\label{acto}
 g:\ (x_0,x_1,y_0,y_1) \mapsto
(a_0x_0+c_0y_0,a_1x_0+a_0x_1+c_1y_0+c_0y_1,
b_0x_0+d_0y_0,b_1x_0+b_0x_1+d_0y_1+d_1y_0).
\end{equation}

Let $l \neq p$ be a prime number.
In the following, 
we analyze the \'{e}tale 
cohomology group $H_c^2(\Tilde{X},
\overline{\mathbb{Q}}_l)$
 on the basis of
  arguments in \cite[Section 3]{Lus}.
To do so, 
as in loc.\ cit.,
 we consider the following 
fibration to the Deligne-Lusztig curve 
$\mathfrak{S}_0:
=\{(x_0,y_0) \in 
\mathbf{F}^2\ |\ x_0y_0^q-x_0^qy_0 \in 
\mathbb{F}^{\times}_q\}$
for $G_1^F$
\begin{equation}\label{fib}
\kappa:\mathfrak{S} \longrightarrow 
\mathfrak{S}_0\ ;\ 
(x_0,x_1,y_0,y_1) 
\mapsto (x_0,y_0).
\end{equation}
Then, we have 
$\kappa(xg)=\kappa(x)\bar{g}$ for 
all $x \in \mathfrak{S}$ 
and $g \in G_2^F.$
The fiber of 
$\kappa$ is 
stable under 
the action of 
$A'$ by (\ref{ooi1}).
Let $\mathfrak{S}
^{\mathbb{F}_q^{\times}}_{00} 
\subset 
\mathfrak{S}_{0}$
be the closed subset 
defined 
by $x_0^{q^2-1}
=y_0^{q^2-1}=-1.$
The set $\mathfrak{S}
^{\mathbb{F}^{\times}_q}_{00}$
is stable under the 
action of $G_1^F$, and 
its action is 
simply transitive.
Hence, the set 
$\mathfrak{S}
^{\mathbb{F}_q^{\times}}_{00}$
consists of 
$|{G}^F_1|=q(q-1)(q^2-1)$
 closed points.
 We set 
 $\xi:=x_0^qy_0-x_0y_0^q.$
Furthermore, we 
put as follows
\[
s_0:=y_0^qx_1-x_0^qy_1,\ 
s_1:=-\xi^{-1}(x_0y_1-x_1y_0),\ 
t_0:=x_0y_0^{q^2}-x_0^{q^2}y_0.
\]
Then, the fiber 
$\kappa^{-1}((x_0,y_0))$ 
for $(x_0,y_0) \in \mathfrak{S}_0$
is identified with the following, 
by a direct computation, 
\begin{equation}\label{fi}
\{(s_0,s_1) 
\in \mathbf{F}^2\ |\
s_0^q+s_0
+s_1^q
t_0 \in 
\mathbb{F}_q\}.
\end{equation}
For $(x_0,y_0) 
\in \mathfrak{S}_0,$
$x_0^{q^2-1}=y_0^{q^2-1}=-1$ is 
equivalent to 
$t_0=0.$
First, let $(x_0,y_0) \in 
\mathfrak{S}_{01}:=\mathfrak{S}_0 
\backslash 
\mathfrak{S}^{\mathbb{F}
_q^{\times}}_{00}.$
Then, we have 
$t_0 \neq 0.$
Hence, by (\ref{fi}), 
the fiber 
$\kappa^{-1}((x_0,y_0))$
is a disjoint union 
of $q$ affine lines.
Secondly, we have 
the following identification 
again by (\ref{fi})
\begin{equation}\label{ls}
\mathfrak{S}_{\ast}:=
\kappa^{-1}(\mathfrak{S}^
{\mathbb{F}_q^{\times}}_{00})
 \simeq 
 \{(x_1,y_1) 
 \in \mathbf{F}^2\ 
 |\ x_1y_0^q-x_0^qy_1=s_0\}_
{s_0 \in \mathbb{F}_{q^2}, (x_0,y_0) 
\in \mathfrak{S}
^{\mathbb{F}^{\times}_q}_{00}} 
\simeq 
\mathbb{A}^1 \times 
{\mathbb{F}_{q^2} 
\times
 \mathfrak{S}^{\mathbb{F}
 ^{\times}_q}_{00}}.
\end{equation}
Furthermore, 
$\mathfrak{S}_{\ast}$ 
is 
stable under 
the actions of 
$G_2^{F}$ and ${\Gamma}.$
We write down 
the action of $G_2^F$ and 
${\Gamma}$ on the fiber space
$\mathfrak{S}_{\ast}$ under
 the identification (\ref{ls}).
For $g 
={\footnotesize \left(
\begin{array}{cc}
a_0+a_1\pi & b_0+b_1\pi \\
c_0+c_1\pi & d_0+d_1\pi
\end{array}
\right)}
\in G_2^F,$
we set
\begin{equation}\label{g_1}
f(i,g):=\{(a_1x_0+c_1y_0)
(b_0x_0+d_0y_0)^q-
(a_0x_0+c_0y_0)^q
(b_1x_0+d_1y_0)\} \in 
\mathbb{F}_{q^2}.
\end{equation}
Note that, if $g \in 
S_2^F$, 
we easily check 
 $f(i,g)^q+f(i,g)=0.$
Then, $g$ acts on 
$s_0 
\in \mathbb{F}_{q^2}$ in (\ref{ls})
as follows, 
by (\ref{acto}), 
\begin{equation}\label{acto3}
g:s_0 \mapsto 
{\rm det}(\bar{g})s_0+f(i,g).
\end{equation} 
On the other hand, 
by (\ref{ooi1}),
 an element 
$\lambda=a_0+a_1\pi 
\in {\Gamma}$ acts 
on 
$s_0 \in \mathbb{F}_{q^2}$
as follows
\begin{equation}\label{tri}
\lambda:s_0 
\mapsto 
a_0^{-(q+1)}
(s_0+(a_1/a_0) \xi).
\end{equation}
We have the 
following 
long exact sequence 
\begin{equation}\label{ex}
\cdots \to H_c^j(
\kappa^{-1}(\mathfrak{S}_{01}),
\overline{\mathbb{Q}}_l)
\to 
H_c^j(\mathfrak{S},
\overline{\mathbb{Q}}_l)
\to H_c^j(\mathfrak{S}_{\ast},
\overline{\mathbb{Q}}_l)
\to \cdots.
\end{equation}
The action of $A'$
on $\kappa^{-1}(\mathfrak{S}_{01})$
 preserves 
each connected 
component of each fiber 
 of the map 
 $\kappa:
 \kappa^{-1}(\mathfrak{S}_{01}) \to 
\mathfrak{S}_{01}.$
Hence, 
$A'$ acts trivially 
on 
$H_c^j(\kappa^{-1}
(\mathfrak{S}_{01}),
\overline{\mathbb{Q}}_l)$ 
for all $j$.
Therefore, by 
(\ref{ex}), 
the restriction 
$H_c^j(\mathfrak{S},
\overline{\mathbb{Q}}_l) 
\to 
H_c^j(\mathfrak{S}_{\ast},
\overline{\mathbb{Q}}_l)$
is an isomorphism on
 the part where we have 
 $\sum_{\lambda \in A'}\lambda=0.$
Since 
$\mathfrak{S}_{\ast}$
is an affine 
line bundle 
over 
$\mathfrak{S}
^{\mathbb{F}
^{\times}_q}_{00} \times 
\mathbb{F}_{q^2},$
 we acquire an isomorphism
\begin{equation}\label{ft}
H_c^2(\mathfrak{S}_{\ast},
\overline{\mathbb{Q}}_l) \simeq 
H^0(\mathfrak{S}
^{\mathbb{F}_q^{\times}}_{00}
 \times \mathbb{F}_{q^2},
 \overline{\mathbb{Q}}_l).
 \end{equation}
Note that we have 
${\rm dim}\ 
H_c^2(\mathfrak{S}_{\ast},
\overline{\mathbb{Q}}_l)^
{\sum_{\lambda \in A'}\lambda=0}
=q^2(q-1)^2(q^2-1).$


In the following, 
we will 
prove that 
the \'{e}tale 
cohomology group 
$H_c^2(\Tilde{X},
\overline{\mathbb{Q}}_{l})$
contains all cuspidal 
representations of $G_2^F$
each one 
with multiplicity $2$ 
in Corollary \ref{Ls}.
More precisely,
 the part  
$H_c^2(\Tilde{X},
\overline{\mathbb{Q}}_l)
^{\sum_{\lambda \in A'}
\lambda=0}=
H_c^2(\mathfrak{S}_{\ast},
\overline{\mathbb{Q}}_l)
^{\sum_{\lambda \in A'}
\lambda=0}$ 
is the direct sum of 
all cuspidal 
representations each
with multiplicity $2$
 similarly as in 
\cite[subsection 3.4]{Lus}.
The part 
$H_c^2(\Tilde{X},
\overline{\mathbb{Q}}_l)_w$
is a cuspidal 
representation if
 $w|_{A'} \neq 1$, 
 which we will 
 give a proof 
 in Proposition 
 \ref{sp} later.
Note that $w \in 
\Gamma^{\vee}$ 
such that $w|_{A'} \neq 1$ 
is a strongly 
primitive character 
in Definition \ref{str}.

To describe 
$H_c^2(\Tilde{X},\overline{\mathbb{Q}}_l)
^{\sum_{\lambda \in A'}\lambda=0},$
we define a 
$G_2^F \times \Gamma$-representation 
${\rho}_{\rm DL}.$ 
We consider $\mathbb{F}^{\vee}_q$
as a subgroup of 
$\mathbb{F}^{\vee}_{q^2}$
by ${\rm Tr}
_{\mathbb{F}_{q^2}/\mathbb{F}_q}^{\vee}.$
For a character $\psi 
\in \mathbb{F}_{q^2}^{\vee}$ 
and $a \in \mathbb{F}_{q^2},$
we write $\psi_a 
\in \mathbb{F}_{q^2}^{\vee}$
 for the character 
$x \mapsto \psi(ax).$
We set as follows 
\begin{equation}\label{kin}
V_{\rho}:=
\bigoplus_{i \in \mathfrak{S}
^{\mathbb{F}^{\times}_{q}}_{00}}\bigoplus_
{\psi \in \mathbb{F}_{q^2}^{\vee} 
\backslash \mathbb{F}_q^{\vee}} 
\overline{\mathbb{Q}}_l
e_{i,\psi}.
\end{equation}
Then, clearly
 we have ${\rm dim}\ V_{\rho}
=q^2(q-1)^2(q^2-1).$
We define an action 
of $G_2^F \times \Gamma$
on the space $V_{\rho}.$
Let $g=
{\footnotesize \left(
\begin{array}{cc}
a_0+a_1\pi & b_0+b_1\pi \\
c_0+c_1\pi & d_0+d_1\pi
\end{array}
\right)}
 \in G_2^{F}.$
 Let $i=(x_0,y_0) \in 
 \mathfrak{S}
 ^{\mathbb{F}_q
 ^{\times}}_{00}.$
 We set $i\bar{g}:
 =(a_0x_0+c_0y_0,b_0x_0+d_0y_0) 
 \in 
 \mathfrak{S}
 ^{\mathbb{F}^{\times}_q}_{00}.$
 Recall $f(i,g)$ in (\ref{g_1}).
Then, 
we define a $G_2^F$-action 
as follows
\begin{equation}\label{g_d}
G_2^F \ni g:\ e_{i,\psi} \mapsto
\psi_{{{\rm det}(\bar{g})}}
(f(i,g^{-1}))
e_{i\bar{g}^{-1},
\psi_{{\rm det}(\bar{g})}}.
\end{equation}  
Note that the $G_2^F$-action 
(\ref{g_d})
on 
$V_{\rho}$ is 
a left action.
Let $t=a_0+a_1\pi 
\in \Gamma$.
We set
$\bar{t}^{-1}i:
=(a_0x_0,a_0y_0) 
\in
 \mathfrak{S}
 ^{\mathbb{F}
 _q^{\times}}_{00}.$
Then, we define a 
$\Gamma$-action 
 as follows
\begin{equation}\label{g_t}
\Gamma \ni t:\ e_{i,\psi} \mapsto 
\psi(-(a_1/a_0)\xi)
e_{\bar{t}^{-1}i,
\psi_{a_0^{-(q+1)}}}.
\end{equation}
We denote this representation 
by $\rho_{\rm DL}.$
Let $\Gamma_{\rm stp}^{\vee} 
\subset \Gamma^{\vee}$
be the subset of
 strongly primitive 
 characters in 
 Definition \ref{str}.
Then, by 
(\ref{acto3}),  
 (\ref{tri}) and 
 (\ref{ft}), 
we acquire
 an isomorphism
\begin{equation}
\label{sio}
\bigoplus_{w 
\in \Gamma^{\vee}_{\rm stp}}
(H_c^2(\Tilde{X},
\overline{\mathbb{Q}}_l)_{w} 
\otimes w) 
=
H_c^2(\Tilde{X}
,\overline{\mathbb{Q}}_l)
^{\sum_{\lambda \in A'}\lambda=0}  
\simeq \rho_{\rm DL}
\end{equation}
as a $G_2^F 
\times \Gamma$-representation.

In the following, we give
 a more concrete 
 description 
of the part 
$H_c^2(\Tilde{X},
\overline{\mathbb{Q}}_l)
_{w}$
for each $w \in 
\Gamma_{\rm stp}^{\vee}.$
To do so, for each $w \in
 \Gamma_{\rm stp}^{\vee},$ 
 we will define 
 a cuspidal representation 
 $\pi_w$ as in 
 \cite[5.2]{AOPS}. 
Now, we fix an 
element $\zeta_0 
\in \mathbb{F}_{q^2} 
\backslash
 \mathbb{F}_q$
and an embedding 
\begin{equation}\label{eb}
\iota_{\zeta_0}:\Gamma=
(\mathcal{O}_E/
\pi^2)^{\times}
\hookrightarrow 
G_2^F
\end{equation}
\[a+b\zeta_0 \mapsto 
a1_{2}+b{\footnotesize 
\left(
\begin{array}{cc}
\zeta_0^q+\zeta_0 & 1 \\
-\zeta_0^{q+1} & 0
\end{array}
\right)}
\]
with $a,b \in 
(\mathcal{O}_F/\pi^2).$
We identify 
$\Gamma \simeq 
\mathbb{F}_{q^2}^{\times} 
\times
\mathbb{F}_{q^2}$
by $a_0+a_1\pi 
\mapsto (a_0,(a_1/a_0)).$
For a character 
$\psi \in 
\mathbb{F}_{q^2}^{\vee} 
\backslash 
\mathbb{F}_q^{\vee}$
and 
an element 
$\zeta 
\in \mathbb{F}_{q^2} 
\backslash \mathbb{F}_q,$ 
we define a 
character 
$\psi_{\zeta}$ of 
$N$ by the following 
\[ 
\psi_{\zeta}:
N \ni {\footnotesize \left(
\begin{array}{cc}
1+\pi a_1 & \pi b_1 \\
\pi c_1 & 1+\pi d_1
\end{array}
\right)
\mapsto 
\psi\biggl(-\frac{a_1\zeta+c_1
-\zeta^q(b_1\zeta+d_1)}
{\zeta^q-\zeta}\biggr).}
\]
Then, by (\ref{eb}),
the restriction 
of $\psi_{\zeta_0} 
\in N^{\vee}$ 
to a subgroup 
$\mathbb{F}_{q^2}
=\Gamma \cap N
\subset N$ is 
equal to $\psi.$
For a strongly 
primitive character 
$w=(\chi,\psi) \in \Gamma^{\vee}
\simeq (\mathbb{F}_{q^2}^{\times})
^{\vee} \times 
\mathbb{F}_{q^2}^{\vee}$
i.e. $\psi \notin 
\mathbb{F}_q^{\vee}$,
we define 
a character 
$w$ of 
$\Gamma N=
\mathbb{F}^{\times}
_{q^2}N$ by
\begin{equation}
\label{cha}
w(xu)=\chi(x)\psi_{\zeta_0}(u)
\end{equation}
for all $x \in 
\mathbb{F}_{q^2}^{\times}$ 
and $u \in N.$
We set
\[
\pi_w:={\rm Ind}
^{G_2^F}_{\Gamma N}(w).
\]
Then, $\pi_w$ is 
a cuspidal 
representation 
of $G_2^F$ by Definition 
\ref{cus}.
Recall that 
we have 
${\rm dim}\ 
\pi_w=q(q-1).$
See 
\cite[subsection 5.2]{AOPS}
 for more details.

Again, we consider 
the $G_2^F
 \times \Gamma$
-representation
 ${\rho}_{\rm DL}.$
 In the following, we give 
 a decomposition of $V_{\rho}$ 
 to
 irreducible omponents.
 To do so, we define 
 several subspaces 
 $\{W_w\}_
 {w \in \Gamma^{\vee}_{\rm stp}}$
 of $V_{\rho},$
 and prove $W_w 
 \simeq 
 \pi_w^{\vee} \times w$
  as a $G_2^F \times \Gamma$
  -representation.
We choose
$w \in \Gamma^{\vee}_{\rm stp}$
 and 
 write 
 $w=(\chi,\psi)$ with $\chi \in 
(\mathbb{F}_{q^2}^{\times})^{\vee}$
and $\psi \in \mathbb{F}_{q^2}^{\vee} 
\backslash 
\mathbb{F}_q^{\vee}.$
We fix an element $y_0 \in \mathbf{F}$ 
such that $y_0^{q^2-1}=-1.$
For each $\zeta \in 
\mathbb{F}_{q^2} 
\backslash 
\mathbb{F}_q$,
we define a 
vector of $V_{\rho}$
in 
(\ref{kin})
\[
e^w_{\zeta}:=
\sum_{\mu \in 
\mathbb{F}_{q^2}}
\chi^{-1}(\mu)
e_{(\zeta\mu y_0,\mu y_0),
\psi_{-((\mu y_0)^{q+1}
(\zeta^q-\zeta))^{-1}}} 
\in V_{\rho}
\]
and set $W^{\zeta}_w:
=\overline
{\mathbb{Q}}_le^w_{\zeta}
 \subset V_{\rho}.$
Furthermore, we put
\[
W_w:=\bigoplus_{\zeta 
\in \mathbb{F}_{q^2} 
\backslash \mathbb{F}_q}
W^{\zeta}_w \subset V_{\rho}.
\]
We can easily check
the following isomorphism
\begin{equation}\label{ve}
V_{\rho} \simeq \bigoplus
_{w \in \Gamma
_{\rm stp}^{\vee}}W_w
\end{equation}
as a 
$\overline{\mathbb{Q}}_l$
-vector space. 
For an element 
$g=
{\footnotesize \left(
\begin{array}{cc}
a_0+a_1\pi & b_0+b_1\pi \\
c_0+c_1\pi & d_0+d_1\pi
\end{array}
\right)} \in G_2^F,$
we obtain the following 
by (\ref{g_d}) 
\begin{equation}\label{poa}
g^{-1} e^w_{\zeta}
=\chi(b_0\zeta+d_0)
\psi_{{\rm det}
(\bar{g})^{-1}}\biggl(
-\frac{f((\zeta 
\mu y_0,\mu y_0),g)}
{(\mu y_0)^{q+1}
(\zeta^q-\zeta)}\biggr)
e^w_{\frac{a_0\zeta+c_0}{b_0\zeta+d_0}}.
\end{equation}
Note that 
$f((\zeta \mu y_0,
\mu y_0),g)/
(\mu y_0)^{q+1}(\zeta^q-\zeta)$ 
is independent 
of $\mu.$
Hence, in particular, 
we obtain  
$g e^w_{\zeta}
=\psi^{-1}_{\zeta}(g)
e^w_{\zeta}$ for
$g \in N$ and $t e^w_{\zeta_0}
=\chi^{-1}(t)e^w_{\zeta_0}$
for $t \in 
\mathbb{F}_{q^2}^{\times} 
\subset \Gamma 
\subset G_2^F$.
On the other hand, 
by (\ref{g_t}), 
we can easily check 
$t e^w_{\zeta}
=w(t)e^w_{\zeta}$
for $t \in \Gamma.$
Hence, by 
(\ref{poa}), 
$W_w$ is a 
$G_2^F \times 
\Gamma$-representation.
By (\ref{poa}),  
the stabilizer of 
the subspace
$W^{\zeta_0}_w$
 in $G_2^F$
is equal 
to $\Gamma N.$
 Moreover, we acquire 
 $W^{\zeta_0}_w \simeq w^{-1}$
 as a 
 $\Gamma N$-representation.
Since 
$G_2^F$ 
permutes 
 the 
 subspaces
  $\{W^{\zeta}_w\}_{\zeta \in
  \mathbb{F}_{q^2}
   \backslash \mathbb{F}_q}$
transitively 
again by (\ref{poa}), we obtain 
$W_w \simeq \pi_w^{\vee} \otimes w$ 
as a $G_2^F \times 
\Gamma$-representation.

Now, we have 
the following 
proposition.
\begin{proposition}\label{sp}
Let the notation 
be as above.
Then, we have 
the followings\ ; 
\\1.\ The following isomorphism 
as a $G_2^F$-representation holds
 \[
 H_c^1(\Tilde{X},
 \overline{\mathbb{Q}}_l)_w \simeq 
 \pi_w^{\vee}.
 \]
\\2.\ We have the 
following isomorphism
\[
{\rho}_{\rm DL} 
\simeq 
\bigoplus_{w \in \Gamma^{\vee}_{\rm stp}}
(\pi_w^{\vee} \otimes w)
\]
as a $G_2^F \times 
\Gamma$-representation.
\end{proposition}
\begin{proof}
The required 
assertion $2$ follows from 
(\ref{ve}) and the isomorphism 
$W_w \simeq \pi^{\vee}_w \otimes w$
as a $G_2^F \times 
\Gamma$-representation.
Therefore, we conclude that 
$H_c^1(\Tilde{X},
\overline{\mathbb{Q}}_l)_w
 \simeq \pi_w^{\vee}$
  by (\ref{sio}). 
  Hence, 
  the assertion $1$
 is proved.
\end{proof}
As a direct consequence 
of Proposition
 \ref{sp}, we obtain 
 the following corollary.
\begin{corollary}\label{Ls}
The representation 
$\rho_{\rm DL}$ contains 
all cuspidal 
representations
each with multiplicity $2.$
\end{corollary}
\begin{proof}
Let $\tau \neq 1 \in 
{\rm Gal}(E/F)$.
Then, 
we have $\pi_{w} 
\simeq \pi_{w^{\tau}}$
as a $G_2^F$-representation 
as mentioned 
 in Remark 
 \ref{wq}.3.
Hence, the 
required assertion 
follows from Proposition
\ref{sp}.2.
\end{proof}
We call $\rho_{\rm DL}$
{\it the unramified 
Lusztig 
represenation
 for $G_2^F.$}

\begin{remark}\label{pl} 
We study the part of 
$H_c^j(\Tilde{X},
\overline{\mathbb{Q}}_l)$
where $A'$ acts trivially.
Clearly, we have 
$H_c^j(\Tilde{X},
\overline{\mathbb{Q}}_l)^{A'=1}
\simeq 
H_c^j(A' \backslash \mathfrak{S},
\overline{\mathbb{Q}}_l)$
for any $j.$
By (\ref{fi}), 
$\mathfrak{S}$ is 
identified with 
$\{(x_0,y_0,s_0,s_1) 
\in \mathfrak{S}_0
 \times \mathbf{F}^2\
 |\ s_0^q+s_0+s_1^qt_0 
 \in \mathbb{F}_q\}.$
Then, $A' \ni 1+a_1\pi$ 
acts on $\mathfrak{S}$
by $(x_0,y_0,s_0,s_1)
 \mapsto 
 (x_0,y_0,s_0+a_1,s_1).$
 Hence, the quotient $A'
  \backslash \mathfrak{S}$
 is isomorphic to a 
 disjoint union of $q$ copies of 
 ${\mathfrak{S}_{0}} \times 
 \mathbb{A}^1 \ni 
 ((x_0,y_0),s_1).$
We denote by $d_2$ 
the composite 
$G_2^F \overset{\rm det}{\to}
(\mathcal{O}_F/\pi^2)^{\times}
\simeq \mathbb{F}^{\times}_q 
\times \mathbb{F}_q 
\overset{{\rm pr}_2}{\to}
\mathbb{F}_q.$
We write 
$t_2$ for the composite 
$\Gamma \simeq 
\mathbb{F}_{q^2}^{\times} 
\times \mathbb{F}_{q^2} 
\overset{{\rm pr}_2}{\to} 
 \mathbb{F}_{q^2} 
 \overset{{\rm Tr}}
 {\longrightarrow }
  \mathbb{F}_q.$
Then, on the set 
$\mathbb{F}_q \ni i$ 
of connected 
components of 
$A' \backslash 
\mathfrak{S}$,
$G_2^F \times \Gamma \ni 
(g,t)$ 
acts by $i 
\mapsto d_2(g)t_2^{-1}(t) i.$
We put 
$M^j:=H_c^{j}(\mathfrak{S}_{0},
\overline{\mathbb{Q}}_l).$
On $M^j,$ the group 
$G_2^F \times \Gamma$
acts by 
factoring through
 $G_2^F \times \Gamma
 \to G_1^F 
\times 
\mathbb{F}
^{\times}_{q^2}.$
Recall that the group
$G_1^F \times 
\mathbb{F}_{q^2}^{\times} \ni
\bigl(g={\footnotesize \left(
\begin{array}{cc}
a_0 & b_0 \\
c_0 & d_0
\end{array}
\right)},\zeta
\bigr)$ acts 
on $\mathfrak{S}_0 \ni 
(x_0,y_0)$
by 
$g:(x_0,y_0) \mapsto
(a_0x_0+c_0y_0,
b_0x_0+d_0y_0)$
and $\zeta:
(x_0,y_0) \mapsto
(\zeta^{-1}x_0,\zeta^{-1}y_0).$
Then, we obtain 
an isomorphism 
$H_c^j
(A' \backslash \mathfrak{S},
\overline{\mathbb{Q}}_l) 
\simeq 
\bigoplus
_{\chi \in 
\mathbb{F}^{\vee}_q}
(M^{j-2} 
\otimes \chi \circ d_2 
\otimes \chi^{-1} 
\circ t_2)$
as a $G_2^F 
\times 
\Gamma$
-representation.
Hence, in particular, 
on the part 
$H_c^3(\Tilde{X},
\overline{\mathbb{Q}}_l)^{A'=1}$, 
 the 
 classical Deligne-Lusztig 
correspondence 
for $G_1^F$ realizes.
\end{remark}

\begin{remark}\label{s_1}
Let $S:={\rm SL}_2(\mathbf{F})$
and $N_0:={\rm Ker}\ 
({S}_2^F 
\to {S}^F_1).$
Clearly, we have an
 isomorphism
$N_0 \simeq 
\{g \in {\rm M}_2(\mathbb{F}_q)\ |\ 
{\rm Tr}_{{\rm M}_2
(\mathbb{F}_q)/\mathbb{F}_q}
(g)=0\}$ as a group, 
 and hence
$N_0$ is 
an abelian group of order $q^3.$
As 
in \cite[3.2]{Lus},
we consider
 the 
 Lusztig surface 
$\Tilde{X}_0:
=\{g \in S_2\ 
|\ gF(g)^{-1} 
\in U_2v\}$ 
for 
$S_2^F.$
More explicitly, 
$\Tilde{X}_0$
is isomorphic to 
the following 
\[
\mathfrak{S}^0:=
\{(x_0,x_1,y_0,y_1) 
\in 
\mathbf{F}^{4}\ |\ 
x_0y_0^q-x_0^qy_0=1,\ 
x_1y_0^q+x_0y_1^q
-x_0^qy_1-x_1^qy_0=0\}
\]
as in 
(\ref{coo1}).
In this case, 
the torus $\Gamma$ 
is replaced by 
$\Gamma_0:=\{(a_0,a_1) 
\in \mathbf{F}^2\ |\ 
a_0^{q+1}=1,\ 
a_0^qa_1+a_0a_1^q=0\}$
of order $q(q+1)$
as in \cite[3.2]{Lus}.
We put  
$\mathcal{I}:={\rm Ker}\ 
({\rm Tr}_{\mathbb{F}_{q^2}
/\mathbb{F}_q}:
\mathbb{F}_{q^2}
 \to \mathbb{F}_q).$
Then, obviously, 
we have isomorphisms
$\Gamma_0 \simeq 
{\rm Ker}\ ({\rm Nr}_{E/F}:\Gamma 
\to 
(\mathcal{O}_F/\pi^2)
^{\times})$ and 
$A' \simeq \mathcal{I}$
as an abelian
 group. 
Furthermore, 
 $\Gamma_0$ contains 
 $A'$ as a subgroup.
 For a character $w
  \in \Gamma_0^{\vee},$ 
  we say that $w$
  is {\it primitive} 
  if $w|_{A'} \neq 1.$ 
  Let
   $\Gamma_{0,{\rm p}}^{\vee} \subset 
  \Gamma_0^{\vee}$ 
  be the subset
   of primitive characters.
   Obviously, 
   we have 
   a natural 
   surjection 
   $\Gamma^{\vee}_{\rm stp}
   \to \Gamma^{\vee}_{0,{\rm p}}.$
Then, the variety 
$\Tilde{X}_0$
 admits an 
action of 
$S_2^F \times 
\Gamma_0$ 
and 
the \'{e}tale 
cohomology group 
$H_c^2(\Tilde{X}_0,
\overline{\mathbb{Q}}_l)$
does so.
Let $\rho^0_{\rm DL}$ denote 
the ${S}_2^F$-representation
$H_c^2(\Tilde{X}_0,
\overline{\mathbb{Q}}_l)
^{\sum
_{\lambda \in A'}
\lambda=0}.$
Note that 
we have ${\rm dim}\ 
\rho^0_{\rm DL}
=q(q-1)(q^2-1).$
The restriction 
of $\iota_{\zeta_0}$ 
in (\ref{eb}) to a
 subgroup $\Gamma_0$ induces 
 an
 embedding
$\Gamma_0 
\hookrightarrow S_2^F$ 
by ${\rm det}
 \circ \iota_{\zeta_0}
 ={\rm Nr}_{E/F}.$
Any character 
$w \in 
\Gamma_{0,{\rm p}}^{\vee}$
extends to the
 character of 
 $\Gamma_0N_0$
uniquely as in 
(\ref{cha}), which
 we denote
  by the same letter $w$.
Then, we define $\pi^0_w:=
{\rm Ind}
^{{S}^F_2}
_{\Gamma_0N_0}(w).$
This representation is 
also called a 
cuspidal 
representation of 
${S}^F_2$
in \cite[3.1]{Sta} and 
all cuspidal 
representations 
arise in this way. 
See also \cite[4.2]{Sha}.
In \cite[p.37]{Sha},
 a cuspidal 
 representation 
 of $S_2^F$
is called an 
{\it unramified} 
representation.
All
 cuspidal 
 representations 
have degree 
$q(q-1)$
and the number 
of them 
is $(q^2-1)/2$
by loc.\ cit.
Similarly as in 
Proposition \ref{sp},
we have the following isomorphism
\[
\rho^0_{\rm DL} \simeq 
\bigoplus_{w \in 
\Gamma^{\vee}_{0,{\rm p}}}
({\pi^0_w}^{\vee} \otimes w).
\]
Hence, as proved 
in \cite[3.4]{Lus}, 
the representation 
$\rho^0_{\rm DL}$
contains all 
cuspidal 
representations 
of $S_2^F$ 
each with 
multiplicity $2.$
\end{remark}

\section{
Review of \cite[7.2]{IMT}
}
In this section, 
we recall the results
in \cite[7.2]{IMT} and 
prove Proposition \ref{gey}
 by using the 
 results in the previous section.
Let $\mathfrak{S}
^{\mathbb{F}^{\times}_q}_{00}$
be as in the previous section.
For $i=(x_0,y_0) \in 
\mathfrak{S}^{\mathbb{F}^{\times}_q}_{00},$
we set $\xi:=x_0^qy_0-x_0y_0^q.$
Then for each $i \in 
\mathfrak{S}^{\mathbb{F}^{\times}_q}_{00}, 
$
let $X_i$ denote the smooth 
compactification of the following 
affine curve $X^{q^2}-X=\xi
(Y^{q(q+1)}-Y^{q+1}).$
Then, $X_i$ has $q$ connected components
and each component 
has genus $q(q-1)/2.$
In loc.\ cit., we prove that 
the components $X_i$
for $i \in
 \mathfrak{S}
 ^{\mathbb{F}^{\times}_q}_{00}$
 apppear in the stable reduction of 
 the Lubin-Tate curve 
 $\mathcal{X}(\pi^2).$
Furthermore, 
in loc.\ cit., 
we analyze the following 
\'{e}tale cohomology group 
\begin{equation}\label{zent}
W :=\bigoplus_{i 
\in \mathfrak{S}^{\mathbb{F}^{\times}_{q}}
_{00}}H^1(X_i,\overline{\mathbb{Q}}_l).
\end{equation}
Then, $W$ has a left action of 
$\mathbf{G}:=G_2^F \times 
\mathcal{O}_3^{\times} 
\times I_F$. 
We consider 
$\Gamma$ as a subgroup of 
$\mathcal{O}_3^{\times}$. 
We will observe that 
the restriction of $W$
to a subgroup $G_2^F 
\times \Gamma \times \{1\}
\subset \mathbf{G}$ 
is related to
the unramified
 Lusztig reprensentation
 $\rho_{\rm DL}$
 computed in the 
 previous section.
The restriction $W|_{G_2^F \times \{1\} 
\times I_F}$ is also written
 with respect to 
 $\rho_{\rm DL}.$
See Proposition \ref{gey} 
for precise statements.

As in the previous section, 
we assume (\ref{asd}).
Since we have $|\mathfrak{S}
^{\mathbb{F}^{\times}_q}_{00}|
=|G_1^F|
=q(q^2-1)(q-1)$ and
${\rm dim}\ H^1(X_i,
\overline{\mathbb{Q}}_l)
=q^2(q-1)$
for each $i$,  
we have ${\rm dim}\ W=q^3(q-1)^2(q^2-1)$.
In the following, we write down 
 the right 
action of $\mathbf{G}$
on the components $\{X_i\}
_{i \in \mathfrak{S}^
{\mathbb{F}^{\times}_{q}}
_{00}}$ given 
in loc.\ cit.
This action induces the 
{\it left} 
$\mathbf{G}$-action 
on $W$ as mentioned above. 
First, we recall the action of 
$G_2^F$. 
Let $g 
={\footnotesize \left(
\begin{array}{cc}
a_0+a_1\pi & b_0+b_1\pi \\
c_0+c_1\pi & d_0+d_1\pi
\end{array}
\right)}
\in G_2^F.$
Then, $g$ acts on 
$\mathfrak{S}^
{\mathbb{F}^{\times}_q}_{00}$ 
as follows, 
factoring through
$G_1^F,$ 
\begin{equation}\label{cvv0}
g:\ i:=(x_0,y_0) \mapsto 
i\bar{g}:=(a_0x_0+c_0y_0,b_0x_0+d_0y_0).
\end{equation}
Furthermore, $g$ induces 
the following morphism
\begin{equation}\label{cvv1}
g:X_i \to 
X_{i\bar{g}}\ ;\ (X,Y) \mapsto 
({\rm det}(\bar{g})X+f(i,g),Y)
\end{equation}
where $f(i,g)$ is 
defined in 
(\ref{g_1}).
Compare (\ref{cvv1}) with 
(\ref{acto3}).
Secondly, 
we recall the action of 
$\mathcal{O}^{\times}_3$. 
Let $b=a_0+\varphi b_0+\pi a_1 
\in \mathcal{O}^{\times}_3$ with 
$a_0 \in 
\mathbb{F}_{q^2}^{\times}$
and $a_1,b_0 \in \mathbb{F}_{q^2}.$
Then, $b$ acts 
on $\mathfrak{S}^{\mathbb{F}^{\times}_q}_{00}$
as follows
\begin{equation}\label{dcc1}
i=(x_0,y_0) \mapsto
 ib:=(a_0^{-1}x_0,a_0^{-1}y_0).
\end{equation}
Moreover, $b$ induces 
a morphism
\[
X_i \to X_{ib}\ ;\ 
(X,Y) \mapsto 
\bigl(a_0^{-(q+1)}
\bigl(X-
(b_0a_0^{-1})\xi 
Y+(a_1a_0^{-1})\xi\bigr),
a_0^{q-1}
\bigl(Y-(b_0a_0^{-1})^q\bigr)\bigr).
\]
In particular, 
the element $t:=a_0+\pi a_1 
\in \Gamma 
\subset 
\mathcal{O}^{\times}_3$ 
acts as follows
\begin{equation}\label{dcc2}
X_i \to X_{it}\ ;\ 
(X,Y) \mapsto 
\bigl(a_0^{-(q+1)}
\bigl(X+(a_1a_0^{-1})\xi\bigr),
a_0^{q-1}Y\bigr).
\end{equation}
Compare this action with (\ref{tri}).
Thirdly, we recall the 
inertia action $I_F.$
We define a homomorphism
\begin{equation}\label{inn}
\mathbf{a}_E:I_F \to 
I_F^{\rm ab} 
\simeq I_E^{\rm ab}
\overset{\mathbf{a}_{E}
|_{I^{\rm ab}_E}}
{\longrightarrow} 
\mathcal{O}^{\times}_E
 \overset{\rm can.}{\to} 
 \Gamma 
\simeq 
\mathbb{F}_{q^2}^{\times}
 \times 
\mathbb{F}_{q^2}\ ;\ 
\sigma \mapsto 
(\zeta(\sigma),\lambda(\sigma)).
\end{equation}
Then, $\sigma 
\in I_{F}$ acts 
on $\mathfrak{S}
^{\mathbb{F}
^{\times}_{q}}_{00}$
as follows
\begin{equation}\label{inn1}
(x_0,y_0) \mapsto 
 i \sigma:
 =(\zeta(\sigma)^{-1}x_0,
 \zeta(\sigma)^{-1}y_0).
\end{equation}
Moreover, $\sigma$ induces a morphism
\begin{equation}\label{inn2}
X_{i} \to X_{i \sigma}\ ;\ 
(X,Y) \mapsto 
(\zeta(\sigma)^{-(q+1)}(X+\lambda(\sigma)\xi),Y).
\end{equation}
The product group
$\mathbb{F}_{q^2} \times \mu_{q+1} \ni (a,\zeta)$
acts on $X_i$ 
with $i 
\in \mathfrak{S}
^{\mathbb{F}^{\times}_q}_{00}$ by 
$(X,Y) \mapsto (X+a,\zeta Y).$
Then, we can decompose 
$H^1(X_i,\overline{\mathbb{Q}}_l)$
to a direct sum of some 
 characters of 
$\mathbb{F}_{q^2} \times \mu_{q+1}$
as follows
\begin{equation}\label{po}
H^1(X_i,\overline{\mathbb{Q}}_l)
\simeq 
\bigoplus_{\chi_0 \in
 \mu_{q+1}^{\vee} 
 \backslash \{1\}}
 \bigoplus_
 {\psi \in 
 \mathbb{F}^{\vee}_{q^2} 
 \backslash 
 \mathbb{F}^{\vee}_q}
 \overline{\mathbb{Q}}_l
 e_{i,\psi,\chi_0}.
\end{equation}
See \cite[Corollary 7.5]{IMT} for 
a proof of (\ref{po}).
By (\ref{po}),
we acquire the following isomorphism
\begin{equation}\label{r_t}
W \simeq \bigoplus_{i \in 
\mathfrak{S}^{\mathbb{F}
^{\times}_q}_{00}}
\bigoplus_{\chi_0 \in \mu_{q+1}^{\vee} 
\backslash \{1\}}
\bigoplus_{\psi 
\in \mathbb{F}
^{\vee}_{q^2} \backslash
 \mathbb{F}_q
 ^{\vee}}\overline{\mathbb{Q}}_l
 e_{i,\psi,\chi_0}
\end{equation}
as a 
$\overline{\mathbb{Q}}_l$-vector space.
Explicit
 descriptions 
 of the 
$\mathbf{G}$-action 
on the right
 hand side of 
 (\ref{r_t}), which is induced 
 by the $\mathbf{G}$-action
 on $W,$ 
are given in \cite[(7.17)-(7.19)]{IMT}.
Now, we write down them.
We have the 
following $G_2^F$-action
\begin{equation}\label{da1}
G_2^F \ni g\ 
:\ e_{i,\psi,\chi_0} \mapsto 
\psi_{{\rm det}(\bar{g})}
(f(i,g^{-1}))
e_{ig^{-1},
\psi_{{\rm det}(\bar{g})},
\chi_0}
\end{equation}
and the following action of 
$\Gamma \subset 
\mathcal{O}_{3}^{\times}$ 
\begin{equation}\label{da2}
\Gamma \ni t=a_0+a_1\pi:
\ e_{i,\psi,\chi_0} \mapsto
\psi(-({a}_1/{a}_0)\xi)
\chi_0^{-1}(t)e_{it^{-1},
\psi_{{a}_0^{-(q+1)}},
\chi_0}.
\end{equation}
Furthermore, we have the 
following $I_F$-action
\begin{equation}\label{da3}
I_F \ni \sigma\ :\
 e_{i,\psi,\chi_0} \mapsto
\psi(-\lambda(\sigma)\xi)e_{i \sigma^{-1}, 
\psi_{\zeta(\sigma)^{-(q+1)}}, 
\chi_0}.
\end{equation}
These descriptions of 
the $\mathbf{G}$-action 
are 
almost direct
 consequences of the 
 $\mathbf{G}$-action
 given 
in (\ref{cvv0})-(\ref{inn2}).
Now, 
we have the following proposition.
\begin{proposition}\label{gey}
Let $\nu:\Gamma \to \mu_{q+1}$ be 
a surjective homomorphism defined by 
$a_0+a_1\pi \mapsto a_0^{q-1}.$ 
For a character 
$\chi_0 \in 
\mu_{q+1}^{\vee},$
we denote by 
the same 
letter $\chi_0$ 
for the composite
$\chi_0 \circ \nu 
\in \Gamma^{\vee}.$
\\1.\ Then, we have 
the following 
isomorphism
\[
W|_{G_2^F \times \Gamma
 \times \{1\}} \simeq 
\bigoplus_{\chi_0 
\in \mu_{q+1}^{\vee}
\backslash \{1\}}
(\rho_{\rm DL} \otimes 
\chi_0^{-1})
\]
as a $G_2^F \times 
\Gamma$-representation.
\\2.\ 
Let $\tilde{\rho}_{\rm DL}$ 
denote the
inflation to $G_2^F \times I_F$ 
of the $G_2^F
 \times \Gamma$-representation
 $\rho_{\rm DL}$
  by a map
 ${\rm id} \times 
 \mathbf{a}_E:
 G_2^F \times 
 I_F 
 \to G^F_2 
 \times \Gamma$.
Then, we have the 
following isomorphism
\[
W|_{G_2^F \times \{1\} 
\times I_F} 
\simeq (\tilde{\rho}
_{\rm DL})^{\oplus q}
\] 
as a $G_2^F \times 
I_F$-representation.
\end{proposition}
\begin{proof}
The required assertions follow 
by comparing 
the $G_2^F \times 
\Gamma$-action (\ref{g_d}) 
and (\ref{g_t}) on 
$V_{\rho}$ (\ref{kin}) 
with the $\mathbf{G}$-action 
 on $W$
 given in
  (\ref{da1})-(\ref{da3}).
\end{proof}
By Proposition \ref{gey},
we acquire the following corollary,
 which is proved in
\cite[Proposition 7.6]{IMT}.
\begin{corollary}\label{cv}
(\cite[Proposition 7.6]{IMT})
We set $\mathbf{H}:=
G_2^F \times \Gamma 
\times I_F \subset \mathbf{G}.$
Then, 
we have the 
following isomorphism
\[
W|_{\mathbf{H}}  \simeq \bigoplus_{w \in 
\Gamma_{\rm stp}^{\vee}}
\biggl(\pi_w^{\vee} \otimes 
\biggl(
\bigoplus_{\chi_0 \in \mu_{q+1}^{\vee} 
\backslash \{1\}}
w \chi_0^{-1}\biggr)  \otimes
w \circ \mathbf{a}_E
\biggr)
\]
as a $\mathbf{H}$-representation.
\end{corollary}
\begin{proof}
The required assertion
 follows from Proposition 
 \ref{sp} 
 and 
 Theorem \ref{gey} 
 immediately.
\end{proof}

\begin{remark}\label{s_2}
Let the notation be 
as in Remark \ref{s_1}.
We briefly remark 
that  
similar things to 
Theorem \ref{gey}
 for 
$S_2^F$ hold.
We set
 $\mathcal{O}
^{1,\times}_3:
={\rm Ker}\ 
({\rm Nrd}_{D/F}:
\mathcal{O}^
{\times}_3 
\to 
(\mathcal{O}_F/
\pi^2)^{\times}).$ 
Then, 
obviously,  
we have a 
natural inclusion 
$\Gamma_0 
\hookrightarrow 
\mathcal{O}^{1,\times}_3$
by ${\rm Nrd}_{D/F}|_{\Gamma}
={\rm Nr}_{E/F}.$
We write 
$\mathbf{a}_F$ for 
the composite 
$I_F \to I^{\rm ab}_F 
\overset{\mathbf{a}_F|
_{I^{\rm ab}_F}}{\to } 
\mathcal{O}^{\times}_F
\overset{\rm can.}{\to } 
(\mathcal{O}_F/\pi^2)
^{\times}.$
Then, we set
$I_F^1:={\rm Ker}\ 
(\mathbf{a}_F:I_F \to
 (\mathcal{O}_F/\pi^2)^{\times}).$
Moreover, 
the restriction 
of the map
 (\ref{inn}) 
 to a 
 subgroup $I_F^1$ 
 induces a 
 surjection 
$\mathbf{a}^0_E:{I}^{1}_F \to
 \Gamma_0$ by 
 $\mathbf{a}_F
 ={\rm Nr}_{E/F} 
 \circ 
 \mathbf{a}_E$.
Let $\mathfrak{S}_{00}:
=\{(x_0,y_0) \in 
\mathbf{F}^2\ |\ 
x_0y_0^q-x_0^qy_0=1,\ 
x_0^{q^2-1}=
y_0^{q^2-1}=-1\}.$
Then, we have 
$|\mathfrak{S}_{00}|
=|S_1^F|=q(q^2-1).$
For each $j 
\in \mathfrak{S}_{00},$
 let $X_j$ be 
 the smooth
 compactification 
 of an affine curve
  $X^q+X+Y^{q+1}=0$ 
  with genus $q(q-1)/2.$
Then, a product group
$\mathcal{I} \times \mu_{q+1} \ni 
(a,\zeta)$
acts on $X_i$ by $(X,Y) 
\mapsto (X+a,\zeta Y).$ 
Then, we consider 
\[
W_0:=
\bigoplus
_{j \in \mathfrak{S}_{00}}
H_c^1(X_j,
\overline
{\mathbb{Q}}_l) \simeq 
\bigoplus_{j \in 
\mathfrak{S}_{00}}
\bigoplus_{\chi_0 
\in \mu_{q+1}^{\vee} 
\backslash \{1\}}
\bigoplus_{\psi 
\in \mathcal{I}^{\vee} 
\backslash 
\{0\}} 
\overline{\mathbb{Q}}
_le_{j,\psi,\chi_0}
\]
where the second 
isomorphism follows from
\cite[Lemma 7.4]{IMT}.
Note that we have 
  ${\rm dim}\ 
  W_0=q^2(q-1)(q^2-1).$
As 
mentioned in section 
\ref{Lo}, 
$f(i,g)$ in
 (\ref{g_1}) is 
 contained in  
$\mathcal{I}$ 
if $g \in S_2^F.$
Similarly, for $\sigma \in I_F^1,$
we have $\zeta(\sigma)^{q+1}=1$
and 
$\lambda(\sigma) 
\in \mathcal{I}.$
For $b=a_0+\varphi b_0+a_1\pi \in 
\mathcal{O}^{1.\times}_3,$
we have $a_0^{q+1}=1$ and 
$(b_0/a_0)^{q+1}
=(a_1/a_0)^q+(a_1/a_0).$
 Hence, 
 a product group 
 $\mathbf{G}_0:=S^F_2 
 \times
  \mathcal{O}^{1,\times}_3 
  \times I_F^1$ acts 
  on the components 
  $\{X_i\}_{i \in 
  \mathfrak{S}_{00}}$ in 
  the same way as 
   (\ref{cvv0})-(\ref{inn2}). 
   Therefore, $\mathbf{G}_0$ 
   also acts on 
   $W_0$ in the same way as 
   (\ref{da1})-(\ref{da3}).  
Then, 
for $W_0$, the same 
statements as in Theorem \ref{gey}
hold by replacing $\rho_{\rm DL}$ by 
${\rho}^0_{\rm DL}$. Namely, we have 
the following isomorphism
$W_0|_{S_2^F \times 
\Gamma_0 \times \{1\}} \simeq 
\bigoplus_{\chi_0 
\in \mu_{q+1}^{\vee} \backslash \{1\}}
(\rho^0_{\rm DL} 
\otimes \chi_0^{-1})$
as a $S_2^F \times \Gamma_0$
-representation.
Let ${\tilde{\rho}^0}_{\rm DL}$
denote the inflation 
of $\rho^0_{\rm DL}$
to $S_2^F \times I_F^1$
by the map 
${\rm id} \times 
\mathbf{a}^0_E:
S_2^F \times I_F^1 \to 
S_2^F \times 
\Gamma_0.$ 
Then, we have 
an isomorphism 
$W_0|
_{S_2^F \times 
\{1\} \times I_F^1}
 \simeq 
 ({\tilde{\rho}^0}_{{\rm DL}})^{\oplus q}$
 as a $S_2^F 
 \times I_F^1$-representation.
Let $\mathbf{H}_0:=S_2^F \times 
\Gamma_0 
\times I_F^1 \subset \mathbf{G}_0.$
Then, similarly as Corollary \ref{cv},
we acquire an isomorphism 
$W_0|_{\mathbf{H}_0} \simeq 
\bigoplus_{w \in \Gamma^{\vee}_{0,{\rm p}}}
\bigl({\pi^0_w}^{\vee} \otimes 
\bigl(\bigoplus
_{\chi \in \mu^{\vee}_{q+1} 
\backslash \{1\}}
w\chi_0^{-1}\bigr) \otimes w 
\circ \mathbf{a}^0_E\bigr)$
as a $\mathbf{H}_0$-representation.
\end{remark}

\section{Main theorem and its proof}
\label{5}
In this section,
we write down the
 Lusztig curve $X_D$ for
 $\mathcal{O}_3^{\times}.$
 The Lusztig curve $X_D$ has an 
 $\mathcal{O}_3^{\times} \times 
 \Gamma$-action.
 We study the cohomology
  group 
  $H_c^1(X_D,
  \overline{\mathbb{Q}}_l)$ 
  as a $\mathcal{O}^{\times}_3
   \times \Gamma$-representation
  in Lemma \ref{lq1}.
We consider the product 
$\mathbf{X}:=\Tilde{X} \times X_D$
with an action of  
$G_2^F \times
 \mathcal{O}_3^{\times} \times \Gamma.$
 The \'{e}tale cohomology group 
 $H_c^3(\mathbf{X},\overline{\mathbb{Q}}_l)$
 is considered as a $\mathbf{G}$-representation
 via the map ${\rm id} 
 \times {\rm id} \times 
 \mathbf{a}_E:\mathbf{G} \to G_2^F \times 
 \mathcal{O}_3^{\times} \times \Gamma.$
Then, in Theorem \ref{fin}, 
we show that 
$W$ is contained in 
 the \'{e}tale
  cohomology group 
 $H_c^3(\mathbf{X},
 \overline{\mathbb{Q}}_l)$
  as a $\mathbf{G}$-subrepresentation. 

As a direct consequence of Corollary \ref{cv},
we prove that 
the following isomorphism 
holds 
\[
W \simeq \bigoplus
_{w \in \Gamma_{\rm stp}^{\vee}}
(\pi_w^{\vee} 
\otimes \rho_w 
\otimes w \circ \mathbf{a}_E)
\]
as a $\mathbf{G}$-representation in
 \cite[Corollary 7.7]{IMT}.
Here, $\rho_w$ with $w=(\chi,\psi)
 \in \Gamma_{\rm stp}^{\vee}
 \simeq (\mathbb{F}^{\times}_q)^{\vee}
  \times 
  (\mathbb{F}^{\vee}_{q^2}
   \backslash \mathbb{F}_q^{\vee})$
is an irreducible 
representation of 
$\mathcal{O}^{\times}_3$
 of degree $q$, 
 which is uniquely 
 characterized by 
 the followings
 $\rho_w|_{U}= 
 \psi^{\oplus q},\ 
 {\rm Tr}\ \rho_w(\zeta)
 =-\chi(\zeta)$ for $\zeta
  \in \mathbb{F}_{q^2} 
  \backslash \mathbb{F}_q,$
  where we set $U:=U_D^2/U_D^3 
  \simeq \mathbb{F}_{q^2} 
  \subset \mathcal{O}_3^{\times}.$
See \cite[Corollary 7.4]{IMT} and 
\cite[Lemma 16.2]{BH}
for more details on $\rho_w.$

In the following, 
we recall the Lusztig curve 
for $\mathcal{O}_3^{\times}$ 
similarly as 
in \cite[Section 2]{Lus2}.
Let ${\footnotesize 
\varphi':=\left(
\begin{array}{cc}
0 & 1 \\
\pi & 0
\end{array}
\right)} \in G_2^F.
$
We define a morphism 
$F':G_2 \to G_2$
by $F'(g)
=\varphi' F(g){\varphi'}^{-1}.$
We define as follows
\[
X_D:=\{g \in G_2\ |\ F'(g)g^{-1} \in U_2\}, 
\]
which we call the Lusztig curve 
for $\mathcal{O}_3^{\times}.$
The condition 
$F'(g)g^{-1} \in 
U_2$ is equivalent to 
$g={\footnotesize \left(
\begin{array}{cc}
\mathbf{x} & y \\
\pi F(y) & F(\mathbf{x})
\end{array}
\right)}
$ with $\mathbf{x} \in A, y
 \in \mathbf{F}$ 
 and ${\rm det}(g)
  \in (\mathcal{O}_F/\pi^2)^{\times}.$
The fixed part $G_2^{F'}$
is equal to the following
$\{[\mathbf{a},b_0]:={\footnotesize \left(
\begin{array}{cc}
\mathbf{a} & F(b_0) \\
\pi b_0 & F(\mathbf{a})
\end{array}
\right)}
\ |\ 
\mathbf{a} \in 
(\mathbb{F}_{q^2}
[[\pi]]/\pi^2)^{\times}, 
b_0 \in \mathbb{F}_{q^2}
\}.$
Hence, we fix the following
 isomorphism 
$G_2^{F'} \overset{\sim }{\to }
\mathcal{O}_3^{\times}\ ;\ 
[\mathbf{a},b_0] \mapsto 
\mathbf{a}+\varphi b_0.$
Then, the group
 $G_2^{F'} \times 
 \Gamma \ni (b,t)$
 acts on $X_D \ni g$ 
 by $g \mapsto 
 {\footnotesize 
 \left(
\begin{array}{cc}
t & 0 \\
0 & F(t)
\end{array}
\right)
}^{-1}gb$ on the right.
We set as follows 
\[
{\mathfrak{S}}_D:
=\{\mathbf{x} \in 
(\mathcal{O}_3 
\otimes_{\mathbb{F}_{q^2}} 
\mathbf{F})^{\times}\ 
|\ {\rm Nrd}_{D/F}(\mathbf{x}) 
\in (\mathcal{O}_F/\pi^2)^{\times}\}.
\]
The group   
 $\mathcal{O}_3^{\times}
  \times \Gamma
 \ni (b,t)$ acts on 
 $\mathfrak{S}_D$
 by
 $\mathbf{x} \mapsto 
 t^{-1}\mathbf{x}b$.
 We acquire an isomorphism
 $\mathfrak{S}_D 
 \overset{\sim}{\to} X_D\ 
 ;\ x_0+y_0\varphi+x_1\pi
 \mapsto {\footnotesize \left(
\begin{array}{cc}
x_0+\pi x_1 & y_0 \\
\pi F(y_0) & F(x_0+\pi x_1)
\end{array}
\right)}$, 
which is compatible
with the $\mathcal{O}_3^{\times} 
\times \Gamma$-actions.
Let $\mathbf{x}=
x_0+y_0\varphi+x_1\pi
 \in \mathfrak{S}_D.$
Then, by 
${\rm Nrd}_{D/F}(\mathbf{x}) \in 
(\mathcal{O}_F/\pi^2)^{\times},$
we acquire 
$x_0 \in 
\mathbb{F}_{q^2}^{\times}$
and $x_0^qx_1+x_0x_1^q-y_0^{q+1}
 \in \mathbb{F}_q.$ 
 We set $\xi:=x_0^{q+1}
  \in \mathbb{F}_q^{\times}.$
 By changing
  variables as follows
 $X=x_0^qx_1$ and $Y=y_0/x_0,$ 
  we obtain
 $X^q+X-\xi Y^{q+1}
  \in \mathbb{F}_q.$
 For each $x_0 \in 
 \mathbb{F}_{q^2}^{\times},$
 let $X_{x_0}$
denote the affine curve 
$X^{q^2}-X
=x_0^{q+1}(Y^{q(q+1)}-Y^{q+1}).$
Then,  
$X_D$ is isomorphic
 to a disjoint union
of $q^2-1$ affine curves
 $\{X_{x_0}\}_{x_0
 \in \mathbb{F}_{q^2}^{\times}}.$
 Furthermore, 
 by a direct computation, 
 $b=a_0+\varphi b_0+\pi a_1 
 \in \mathcal{O}_3^{\times}$ 
 induces 
 the following morphism
\begin{equation}\label{b_1}
X_{x_0} \to X_{a_0x_0}\ :\ 
 (X,Y) \mapsto \bigl(a_0^{q+1}
 \bigl(X+(b_0/a_0)
 \xi Y+(a_1/a_0)\xi
 \bigr),a_0^{q-1}(Y+(b_0/a_0)^q)\bigr).
 \end{equation}
 On the other hand, 
 $\Gamma \ni t=a_0+a_1\pi$
 induces the following map
 \begin{equation}\label{b_2}
 X_{x_0} \to X_{a_0^{-1}x_0}\ :\
(X,Y) \mapsto 
(a_0^{-(q+1)}
(X-(a_1/a_0)\xi),Y).
 \end{equation}
\begin{lemma}\label{lq1}
Let the notation be as above.
Then,
 we have the following 
isomorphism 
$H_c^1(X_D,\overline{\mathbb{Q}}_l) \simeq 
\bigoplus_{w \in 
\Gamma_{\rm stp}^{\vee}}
(\rho_w \otimes w^{-1})$
as a 
$\mathcal{O}^{\times}_3 
\times \Gamma$-representation.
\end{lemma}
\begin{proof}
Similarly as in
 section \ref{Lo},
we prove the following 
isomorphism by using 
(\ref{po})
\[
H_c^1(X_D,\overline{\mathbb{Q}}_l)|_{\Gamma \times \Gamma}
\simeq \bigoplus_{w \in \Gamma_{\rm stp}^{\vee}}
\biggl(\bigl(\bigoplus_{\chi_0\in 
\mu^{\vee}_{q+1} \backslash \{1\}}w\chi_0\bigr) 
\otimes w^{-1}\biggr)
\]
as a $\Gamma \times \Gamma$-representation.
Thereby, 
for $w=(\chi,\psi) 
\in \Gamma_{\rm stp}^{\vee},$ 
$\rho^w
:=H_c^1(X_D,\overline{\mathbb{Q}}_l)_{w^{-1}}$
satisfies 
$\rho^w|_{U}
=\psi^{\oplus q}$ and ${\rm Tr}\
 \rho^w(\zeta)=-\chi(\zeta)$ for $\zeta 
 \in \mathbb{F}_{q^2} 
 \backslash \mathbb{F}_q.$
 Hence, we acquire $\rho^w \simeq \rho_w.$
\end{proof}

Let $\Tilde{X}$ be as 
in section \ref{Lo}.
Now, 
we consider the fiber product 
$\mathbf{X}:=\Tilde{X} \times X_D.$
On this variety, 
$G_2^F \times \mathcal{O}_3^{\times} 
\times \Gamma \times \Gamma$ acts.
We restrict the $\Gamma \times \Gamma$-action 
on $\mathbf{X}$
to a subgroup 
$\Gamma 
\hookrightarrow 
\Gamma \times \Gamma:t 
\mapsto (t,t^{-1}).$
Then, the product 
$G_2^F 
\times \mathcal{O}_3^{\times} \times \Gamma$
acts on 
the variety $\mathbf{X}.$
Hence, let $\mathbf{G}$
act on $\mathbf{X}$ according to 
${\rm id} \times {\rm id} \times \mathbf{a}_E
:\mathbf{G} \to G_2^F \times
 \mathcal{O}_3^{\times} \times \Gamma.$
Now, we consider $H_c^3
(\mathbf{X},\overline{\mathbb{Q}}_l)$
as a $\mathbf{G}$-representation. 
For each $w \in 
\Gamma_{\rm stp}^{\vee},$
let $H_c^3
(\mathbf{X},\overline{\mathbb{Q}}_l)_w$
be the subspace of 
$H_c^3(\mathbf{X},\overline{\mathbb{Q}}_l)$
 on which 
$I_F$ acts according to
 $w \circ \mathbf{a}_E.$
\begin{theorem}\label{fin}
Let the notation be as above. Then, we have the followings:\ 
\\1.\ For each $w \in \Gamma_{\rm stp}^{\vee},$ 
the following isomorphism holds
$H_c^3(\mathbf{X},
\overline{\mathbb{Q}}_l)_w \simeq \pi_w^{\vee} 
\otimes \rho_w \otimes w \circ \mathbf{a}_E$ as 
a $\mathbf{G}$-representation.
\\2.\ We have the following isomorphism
\[W
\simeq 
\bigoplus_{w \in \Gamma_{\rm stp}^{\vee}}
H_c^3(\mathbf{X},
\overline{\mathbb{Q}}_l)_w\]
as a $\mathbf{G}$-representation.
\end{theorem}
\begin{proof}
The assertion $2$ follows from $1$
immediately. We prove 
the assertion $1.$
Since 
$\mathfrak{S}_{\ast}$ is an affine 
line bundle over a finite set, 
 we obtain $H_c^1(\Tilde{X},
\overline{\mathbb{Q}}_l)
^{\sum_{\lambda \in A'}\lambda=0}
\simeq H_c^1(\mathfrak{S}_{\ast},
\overline{\mathbb{Q}}_l)
^{\sum_{\lambda \in A'}\lambda=0}=0$.
On the other hand, we have 
$H_c^1(\Tilde{X},
\overline{\mathbb{Q}}_l)^{A'=1}=0$
by Remark \ref{pl}.
Hence, we acquire $H_c^1(\Tilde{X}
,\overline{\mathbb{Q}}_l)=0.$
Therefore, the following holds  
$H_c^3(\mathbf{X},
\overline{\mathbb{Q}}_l) \simeq 
H_c^2(\Tilde{X},\overline{\mathbb{Q}}_l)
 \otimes H_c^1(X_D,\overline{\mathbb{Q}}_l)$
 by the Kunneth formula.
Hence, for each $w \in \Gamma^{\vee}
_{\rm stp}$, 
the part $H_c^3(\mathbf{X},
\overline{\mathbb{Q}}_l)_w$ 
is isomorphic to 
$\pi_w^{\vee} \otimes \rho_w$
as a $G_2^F \times 
\mathcal{O}_3^{\times}$-representation
by Proposition \ref{sp} and 
Lemma \ref{lq1}.
Therefore, the required assertion follows.
\end{proof}
\begin{remark}
Let the notation be
 as in Remark \ref{s_2}.
Let 
$X_{0,D}:=\{g \in S_2\ |\ F'(g)g^{-1} \in U_2\}.$
This is the Lusztig curve for 
$\mathcal{O}_3^{1,\times}$,
 which 
is computed in 
\cite[Section 2]{Lus2}.
Then, 
$X_{0,D}$ 
is isomorphic to a
 disjoint union of 
$q+1$ copies of 
an affine curve $X^q+X=Y^{q+1},$
for which we write 
$\{X_{x_0}\}_{x_0 \in \mu_{q+1}}.$
The group $\mathcal{O}_3^{1,\times}
 \times \Gamma_0$
acts on $\{X_{x_0}\}_{x_0 \in \mu_{q+1}}$
in the same way as (\ref{b_1}) 
and (\ref{b_2}).
We set $U_0:=U \cap
 \mathcal{O}_3^{1,\times}
  \simeq \mathcal{I}.$
Then, as in Lemma \ref{lq1}, we acquire 
$H_c^1(X_{0,D},\overline{\mathbb{Q}}_l)
\simeq 
\bigoplus_{w \in \Gamma^{\vee}_{0,{\rm p}}}
(\rho_w^0 \otimes w^{-1})$
as a $\mathcal{O}^{1,\times}_3 
\times \Gamma_0$-representation.
Here, for   
$w=(\chi,\psi) \in \Gamma_{0,{\rm p}}^{\vee}
\simeq \mu_{q+1}^{\vee}
 \times (\mathcal{I}^{\vee}
  \backslash \{0\}),$ 
   $\rho_w^0$
  is an irreducible representation of 
$\mathcal{O}^{1,\times}_3$ of degree $q$
such that $\rho_w^0|_{U_0} 
=\psi^{\oplus q}$
 and ${\rm Tr}\ \rho_w^0(\zeta)=-\chi(\zeta)$
 for $\zeta \in \mu_{q+1} \backslash \{1\}.$
We consider $\mathbf{X}_0:=\Tilde{X}_0 \times X_{0,D}$
with a right $S_2^F \times \mathcal{O}_3^{1,\times}
\times \Gamma_0 \times \Gamma_0$.
We restrict the 
$\Gamma_0 \times 
\Gamma_0$-action on $\mathbf{X}_0$
to a subgroup
 $\Gamma_0 \hookrightarrow
 \Gamma_0 \times \Gamma_0:t \mapsto (t,t^{-1}).$
By a surjective map 
${\rm id} \times 
{\rm id} \times \mathbf{a}_E^0
:\mathbf{G}_0 \to S_2^F \times \mathcal{O}_3^{1,\times}
\times \Gamma_0$, $H_c^3(\mathbf{X}_0,\overline{\mathbb{Q}}_l)$
is considered as a $\mathbf{G}_0$-representation.
Then, similarly as Theorem \ref{fin}, 
we have the following isomorphisms 
$H_c^3(\mathbf{X}_0,\overline{\mathbb{Q}}_l)_w
\simeq {\pi^0_w}^{\vee} \otimes \rho_w^0
 \otimes w \circ \mathbf{a}_E^0
$ for each $w \in \Gamma_{0,{\rm p}}^{\vee}$
 and $W_0 \simeq 
 \bigoplus_{w \in 
\Gamma_{0,{\rm p}}^{\vee}}
H_c^3(\mathbf{X}_0,\overline{\mathbb{Q}}_l)_w$ 
as $\mathbf{G}_0$-representations.
\end{remark}

\noindent
Tetsushi Ito\\ 
Department of Mathematics, Faculty of Science,
Kyoto University, Kyoto 606-8502, Japan\\ 
tetsushi@math.kyoto-u.ac.jp\\ 

\noindent
Yoichi Mieda\\ 
Faculty of Mathematics,
Kyushu University, 744 Motooka, Nishi-ku, 
Fukuoka city, 
Fukuoka 819-0395, Japan\\ 
mieda@math.kyushu-u.ac.jp\\

\noindent
Takahiro Tsushima\\ 
Faculty of Mathematics,
Kyushu University, 744 Motooka, Nishi-ku, 
Fukuoka city, 
Fukuoka 819-0395, Japan\\
tsushima@math.kyushu-u.ac.jp\\

\end{document}